\documentclass [reqno,11pt,oneside]{amsart}
\usepackage[latin1]{inputenc}
\usepackage{amsfonts}
\usepackage{amsmath}
\usepackage{amsthm}
\usepackage{amssymb}
\usepackage{cite}
\usepackage{geometry}
\usepackage{verbatim}

\usepackage{graphicx}

\renewenvironment {enumerate}%
  {\rule{1mm}{0mm}\begin {oldenumerate}%
    \parskip1ex plus0.5ex \itemsep 0mm \parindent 0mm}%
  {\end {oldenumerate}}

\renewenvironment {itemize}%
  {\rule{1mm}{0mm}\begin {olditemize}%
    \parskip1ex plus0.5ex \itemsep 0mm \parindent 0mm}%
  {\end {olditemize}}

\theoremstyle {plain}
\newtheorem {thm}{Theorem}[section]

\newtheorem {lem}[thm]{Lemma}

\theoremstyle {definition}
\newtheorem {defi}[thm]{Definition}
\theoremstyle {remark}
\newtheorem {rem}[thm]{Remark}

\newtheorem {example}[thm]{Example}

\newtheorem {convention}[thm]{Convention}
\newtheorem*{acknowledgement*}{Acknowledgement}
\newcommand\LOT{\ensuremath{\textrm{LOT}}}
\newcommand\HT{\ensuremath{\textrm{HT}}}
\newcommand\HC{\ensuremath{\textrm{HC}}}

\newcommand\MHT{\ensuremath{\textrm{MHT}}}

\newcommand\MHTF{\ensuremath{\textrm{MHT}_{\textrm{\tiny{F}}}}}

\newcommand\crit{\ensuremath{\textrm{crit}}}
\newcommand\rew{\ensuremath{\textrm{rew}}}
\newcommand\red{\ensuremath{\textrm{red}}}
\newcommand\spp{\ensuremath{\textrm{sp}}}
\newcommand\prev{\ensuremath{\textrm{prev}}}
\newcommand\me{\ensuremath{\mathbf{e}}}
\newcommand\mg{\ensuremath{\mathbf{g}}}
\newcommand\ms{\ensuremath{\mathbf{s}}}
\newcommand\ind{\ensuremath{\textrm{index}}}
\newcommand\poly{\ensuremath{\textrm{poly}}}
\newcommand\Spol{\ensuremath{\textrm{Spol}}}
\newcommand\Sig{\ensuremath{\mathcal{S}}}

\newcommand\Kx{\ensuremath{\mathbb{K}[\underline{x}]}}
\newcommand\Kxm{\ensuremath{\mathbb{K}[\underline{x}]^m}}
\newcommand\KxnG{\ensuremath{\mathbb{K}[\underline{x}]^{n_G}}}

\newcommand\preceqf{\ensuremath{\preceq_{\textrm{F}}}}

\newcommand\precf{\ensuremath{\prec_{\textrm{F}}}}
\newcommand\succf{\ensuremath{\succ_{\textrm{F}}}}

\newcommand\LCM{\ensuremath{\textrm{LCM}}}

\begin{document}
\bibliographystyle{alpha}
\title{On the Criteria of the $F_5$ Algorithm}
\author{Christian Eder}
\address {Christian Eder, Fachbereich Mathematik, TU Kaiserslautern, Postfach 3049, 67653 Kaiserslautern, Germany}
\email{ederc@mathematik.uni-kl.de}
\date{}

\maketitle
\begin{abstract}
Faug\`{e}re's $F_5$ algorithm is one of the fastest known algorithms for the computation of Gr\"obner bases. So far only the $F_5$ Criterion is proved, whereas the second powerful criterion, the Rewritten Criterion, is not understood very well until now. We give a proof of both, the $F_5$ Criterion and the Rewritten Criterion showing their connection to syzygies, i.e. the relations between the S-Polynomials to be investigated by the algorithm. Using the example of a Gr\"obner basis computation stated in \cite{f5} we show how Faug\`{e}re's criteria work, and discuss the possibility of improving the $F_5$ Criterion.  
\end{abstract}

\section{Introduction}
The $F_5$ algorithm stated in 2002 in \cite{f5} is one of the fastest Gr\"obner basis algorithms up to date, but there are still not many implementations due to problems understanding the algorithm and its criteria to detect useless critical pairs of polynomials. \\
There are two main criteria: The $F_5$ Criterion and the Rewritten Criterion. Whereas proofs of the $F_5$ Criterion are given in \cite{f5} and later on in a slightly different way also in \cite{f5rev} there is still no proof for the Rewritten Criterion\footnote{Recently Gash has given another proof of the Rewritten Criterion in \cite{gash}}. Stegers tries to give an idea of how the criterion works, but he is not able to give a full proof.\\
In this paper we prove the correctness of both criteria and show that both are based on a similar relation between syzygies and interdependent S-Polynomials. Tightening the insight of the two criteria by giving examples and constructing the relations between the S-Polynomials using the ideas of the proof, this leads to an idea of an improvement of the $F_5$ Criterion also. We show that this improvement is not possible and there cannot be a generalization of the $F_5$ Criterion. Afterwards we explain the problem of connecting the discussed criteria with the 1st and 2nd Buchberger Criterion. This problem is strongly related to the dependence of Faug\`{e}re's criteria on the signatures, whereas the Buchberger criteria do only care about the polynomial part of the critical pairs investigated.\\
The plan of this paper is the following: In Section \ref{sec:basics} we give basic notations and definitions used in the $F_5$ algorithm. Section \ref{sec:faugere} includes the main theorem of this paper, Theorem \ref{thm:main} in whose proof the correctness of both, the $F_5$ Criterion and the Rewritten Criterion is shown. In the following we give for each criterion 3 detailed examples of how to use the constructive proof of Theorem \ref{thm:main} to see the correctness of deleting the detected pairs in the example given in \cite{f5} Section 8. Afterwards we discuss the question of improving the $F_5$ Criterion on the basis of the constructive proof of the main theorem in Section \ref{sec:improvement} and show its failure. In the Appendix a short note on the current $F_5$ implementation in the computer algebra system \textsc{Singular} is given. \\
Note that in this paper we do not state or prove the correctness or termination of any of the mentioned algorithms, we just prove the correctness of their criteria used, not the correctness and termination of the algorithms/implementations.\\
The proofs of the criteria are a joint work with John Perry. This paper represents my version of the results of our work.
Another paper, which will include the discussion of the criteria as well as our discussion of the termination and correctness of $F_5$, is in preparation by John Perry.
\begin{acknowledgement*}
\end{acknowledgement*}
The proofs of the criteria and the implementation of F5 in a Singular library are joint work with John Perry.

\section{Basic Concepts}
\label{sec:basics}
First of all we need to state and understand the main definitions of Faug\`{e}re's approach to work with polynomials during Gr\"obner bases computations. For this we need to find a relation between polynomials and module elements corresponding to them. This relation adds a new information to the polynomial which is later on used to decide if it is useful or not for the computation of a Gr\"obner basis.
\subsection{Connection Between Polynomials And Module Elements}
We state the main ideas of \cite{f5} whereas we rewrite them in a slightly different way for the sake of simplicity.

\begin{convention}
In the following $K$ is always a field, $\underline{x} = (x_1,\dots,x_n)$, $\mathcal{T}$ denotes the set of terms of the ring $\Kx$. 
Let $F=(f_1,\dots,f_m)$ be a sequence of polynomials $F_i \in \Kx$ for $i \in \{1,\dots,m\}$ such that $I = \langle f_1,\dots,f_m \rangle$.  
Let $<$ denote a term order on $\Kx$. \\
Let $p_1, p_2 \in \Kx$, $u_k = \frac{\LCM(\HT(p_1),\HT(p_2))}{\HT(p_k)}$ for $k \in \{1,2\}$ then we denote the S-Polynomial of $p_1$ and $p_2$ $\Spol(p_1,p_2) =\HC(p_2) u_1 p_1 - \HC(p_1) u_2 p_2$.
\end{convention}

\begin{defi}
\label{def:faugere}
\begin{enumerate}

\item Let $\Kxm$ be an $m$-dimensional module with generators $\me_1,\dots,\me_m$. 
Elements of the form $t \me_i$ such that $t \in \mathcal{T} \subset \Kx$ are called \emph{module terms}. We define the \emph{evaluation map} 
\begin{eqnarray*}
v_F : \Kxm &\rightarrow& \Kx \\
     \me_i &\mapsto& f_i \quad \textrm{for all } i \in \{1,\dots,m\}. 
\end{eqnarray*}
A \emph{syzygy} of $\Kxm$ is an element $\ms \in \Kxm$ such that $v_F(\ms) = 0$.
\item We define the module term ordering $\precf$ on $\Kxm$:
\begin{eqnarray*}
t_i \me_i \precf t_j \me_j   :\Leftrightarrow &(a)& i>j, \textrm{ or} \\
                                              &(b)& i = j \textrm{ and } t_i<t_j
\end{eqnarray*}

\item For an element $\mathbf{g} = \sum_{i=1}^m \lambda_i \me_i \in \Kxm$ we define the \emph{index of $\mathbf{g}$} $\ind(\mathbf{g})$ to be the lowest number $i_0$ such that $\lambda_{i_0} \neq 0$. Let $\ind(\mathbf{g}) = k$, then the module head term of $\mathbf{g}$ w.r.t. $F$ is defined to be $\MHTF(\mathbf{g}) = \HT(\lambda_k) \me_k$.

\item Let $p \in \Kx$ be a polynomial, we call $p$ \emph{admissible w.r.t. $F$} if there exists an element $\mathbf{g} \in \Kxm$ such that $v_F(\mathbf{g}) = p$.

\item 
\label{def:sig}
A \emph{admissible w.r.t. $F$, labeled polynomial $r$} is an element of $\Kxm \times \Kx$ defined by 
\begin{equation*}
r=\big(\Sig(r),\poly(r)\big)
\end{equation*}
where the components of $r$ are defined as follows:
\begin{enumerate}
\item $\poly(r) \in \Kx$ denotes the \emph{polynomial part of $r$}.
 $\Sig(r)$ denotes the \emph{signature of $r$} and is defined to be
\begin{equation*}
\Sig(r) = \MHTF(\mathbf{g}) \textrm{ such that }v_F(\mathbf{g}) = poly(r).
\end{equation*}  
\item The \emph{index of $r$}, $\ind(r)$ is defined to be $\ind(\mathbf{g})$ where
\begin{equation*}
\MHT(\mathbf{g}) = \Sig(r) \textrm{ and } v_F(\mathbf{g}) = \poly(r). 
\end{equation*}
\end{enumerate}
\item Let $r$ be an admissible w.r.t. $F$, labeled polynomial such that $\Sig(r) = t_i \me_i$. Then we define the \emph{term of the signature} to be 
\begin{equation*}
\Gamma(\Sig(r)) = t_i.
\end{equation*}

\item Let $r_1=\big(\Sig(r_1),\poly(r_1)\big)$ and $r_2=\big(\Sig(r_2),\poly(r_2)\big)$ be two admissible labeled polynomials such that $u_2 \Sig(r_2) \precf u_1 \Sig(r_1)$. 
\label{def:sigspoly}
Then 
\begin{equation*}
\Spol(r_1,r_2) = \Big( u_1 \Sig(r_1), \Spol \big(\poly(r_1),\poly(r_2)\big) \Big) 
\end{equation*}
\end{enumerate}
\end{defi}

\begin{rem}
\label{rem:sig}
\begin{enumerate}

\item The notations $\MHTF$ and $\precf$ are due to distinguish Faug\`{e}re's definition of a module term ordering in \cite{f5} with the same approach in a different way of M\"oller, Traverso, and Mora in \cite{traverso}, on which Faug\`{e}re's ideas finding useless critical pairs is based on. \\
Note that the index F of $\MHTF$ does not belong to the sequence $F$ of polynomials in $\Kxm$ also.

\item Note that the definition of the signature in \ref{def:faugere}\ref{def:sig} is different from Faug\`{e}re's one in \cite{f5}. Our understanding of a signature of a labeled polynomial $r$ is equal to Faug\`{e}re's definition of an admissible labeled polynomial $r$.
This is due to the fact that the origin definition of the signature is not useful in the concept of computing Gr\"obner bases. 
Beside from Proposition 1 and Corollary 1 Faug\`{e}re does not use his definition of the signature. When computing the Gr\"obner basis with the $F_5$ algorihm signatures are computed in the sense of Definition \ref{def:faugere}\ref{def:sig}, hence we do not refer to Faug\`{e}re's initially definition when speaking of the signature of an admissible w.r.t. $F$ labeled polynomial, but to the definition given in this paper.

\item Note moreover that the signature $\Sig(r)$ of an adsmissible w.r.t. $F$, labeled polynomial $r$ by Definition \ref{def:faugere}\ref{def:sig} is not uniquely defined. 
\end{enumerate}
\end{rem}

\begin{example}
Assume the sequence $F=(f_1,\ldots,f_m)$.
\begin{enumerate}
\item[(a)] Let $p = f_1$. Then $r=(\me_1,f_1)$ is an admissible labeled polynomial as $v_F(\me_1)=f_1$.
\item[(b)] Again let $p=f_1$. Then $r'=(f_2 \me_1, f_1)$ is also an admissible labeled polynomial.
For this consider the module element $\mg=(f_2 + 1) \me_1 - f_1 \me_2$. It holds that $v_F(\mg)= f_2 f_1 + f_1 - f_1 f_2 = f_1$ and $\MHT(\mg)=f_2 \me_1$.
\end{enumerate}
\end{example}

\begin{rem}
The $F_5$ Algorithm always takes the minimal possible index at the given iteration step during its computations. 
In the above situation the $F_5$ Criterion (see Definition \ref{def:f5crit}) would detect and delete $r'$. 
This is an important point as in the case of $F$ being a regular sequence all of these multiple descriptions of the signature can be detected and only the in some sense minimal one remains in the computations. Thus the signature computed by $F_5$ is unique in the case of an regular input.
\end{rem}
\newpage
\begin{convention}
\begin{enumerate}
\item Due to the fact that in the following all labeled polynomials will be admissible w.r.t. $F$, we drop the reference to which set the admissibility is referred to for a shorter notation. 
\item Let $r$ be an admissible labeled polynomial. For a better legibility let in the following always denote $p = \poly(r)$. So when referring to the signature and admissibility of an element we use the letter $r$, i.e. the labeled polynomial in $\Kxm \times \Kx$, when considering the computations in terms of the polynomials itself we use the letter $p$, i.e. the polynomial in $\Kx$.
\end{enumerate}
\end{convention}

\subsection{The Relation To Computations Of Gr\"obner Bases}
To understand the two main criteria of the $F_5$ algorithm we embed $\Kxm$ into the module $\KxnG$ in a canonical way, i.e. $n_G\geq m$ and $\KxnG = \Kxm \times \mathbb{K}[\underline{x}]^{n_G-m}$.

\begin{convention}
\label{conv:embedding}
In the following $G=\{r_1,\dots,r_{n_G}\}$ always denotes a set of admissible labeled polynomials such that $\poly(G) := \{ p_i \mid r_i \in G \} \supset \{f_1,\dots,f_m\}$. 
We assume that $r_i = (\me_i, f_i)$ for all $i \in \{1,\dots,m\}$ for the rest of this paper.
\end{convention}

\begin{defi}
\begin{enumerate}
\item We define an evaluation map 
\begin{eqnarray*}
v_G : \KxnG &\rightarrow& \Kx \\
     \me_i &\mapsto& p_i \quad \textrm{for all } i \in \{1,\dots,n_G\}. 
\end{eqnarray*}
A \emph{syzygy} of $\KxnG$ is an element $\ms \in \KxnG$ such that $v_G(\ms) = 0$.

\item For each $\me_i$ where $i \in \{1,\dots,n_G\}$ we define the module head term to be 
\begin{equation*}
\MHTF(\me_i) = \Sig(r_i)
\end{equation*}
as defined in \ref{def:faugere}\ref{def:sig} and \ref{def:faugere}\ref{def:sigspoly}.
\end{enumerate}
\end{defi}

\begin{rem}
Note that by Convention \ref{conv:embedding} $v_F(\me_i) = v_G(\me_i)$ for all $i \in \{1,\dots,m\}$.
\end{rem}

Using admissible labeled polynomials to describe Gr\"obner bases for given ideals we need to define an admissible labeled equivalent to the $t$-representation known for polynomials in $\Kx$:

\begin{defi}
\label{def:admrep}
Let $r=(\Sig(r),p)$ be an admissible labeled polynomial, $\mathcal{M} = \{ r_1,\dots,r_{n_\mathcal{M}}\}$ be a set of admissible labeled polynomials, and $t = \HT(p)$.
A representation 
\begin{equation*}
p = \sum_{j = 1}^{n_{\mathcal{M}}} \lambda_j p_j, \quad \lambda_j \in \Kx
\end{equation*}
is an \emph{admissible labeled $t$-representation of (the admissible labeled polynomial) $r$} if $\HT(\lambda_j p_j) < t$ and $\HT(\lambda_j) \Sig(r_j) \preceqf \Sig(r)$ for all $j$. 
\end{defi}

There is an easy connection between usual and admissible labeled $t$-representations: 

\begin{lem}
\label{lem:rep}
Let $r$ be an admissible labeled polynomial. If $r$ has an admissible labeled $t$-representation for $t = \HT(p)$ then $p$ has a $t$-representation.
\end{lem}

\begin{proof}
Clear by Definition \ref{def:admrep}.
\end{proof}

\begin{convention}
When speaking of an admissible labeled $t$-representation of an S-Polynomial $\Spol(r_i,r_j)$ in the following without explicitly denoting $t$ we always assume that $t=\LCM\big(\HT(p_i),\HT(p_j)\big)$.
\end{convention}

It follows that we can give a new characterization of a Gr\"obner basis using admissible labeled polynomials.

\begin{thm}
\label{thm:admgroeb}
If for all elements $r_i, r_j \in G$ $\Spol(r_i,r_j)$ has an admissible labeled $t$-representation or $\Spol(p_i,p_j)$ reduces to zero then $\poly(G)$ is a Gr\"obner basis of $I=\langle f_1,\dots,f_m \rangle$.
\end{thm}

\begin{proof}
Clear by the characterization of a Gr\"obner basis and Lemma \ref{lem:rep}.
\end{proof}

\section{Faugere's Criteria}
\label{sec:faugere}
Whereas a Gr\"obner basis $G$ can be characterized by Theorem \ref{thm:admgroeb} it does not improve its computation, on the contrary we require even more, the polynomials need to be labeled and admissible w.r.t. a given set and their representations need to fulfill another criterion on their signatures. 
As the $F_5$ algorithm constructs new elements exactly such that they have admissible labeled $t$-representations, Faug\`{e}re uses two criteria to check if the S-Polynomial of a critical pair needs to be computed and reduced, or if the critical pair is useless for the computation of $G$. \\ 
To decide if one of the criteria holds, the signatures of the labeled polynomials are used. By this means Faug\`{e}re uses these new requirements on an admissible labeled $t$-representation stated in the previous section to get information on the relations between S-Polynomials which help to decide the necessity of them. \\
We state these criteria and prove their correctness, but we do not explain the $F_5$ algorithm in detail, we refer to \cite{f5} or \cite{f5rev} for a deeper insight in $F_5$.

\begin{defi}[$F_5$ Criterion]
\label{def:f5crit}
Let $(r_i,r_j) \in G \times G$ be a critical pair. $\Spol(r_i,r_j)$ is \emph{not normalized} iff for $u_k r_k$, $k=i$ or $k=j$, there exists $r_\prev \in G$ such that 
\begin{eqnarray*}
\ind(r_\prev) &>& \ind(r_k) \textrm{ and} \\
\HT(p_\prev) &\mid& u_k \Gamma\big(\Sig(r_k) \big) 
\end{eqnarray*} 
If there exists no such $r_\prev \in G$ then $\Spol(r_i,r_j)$ is \emph{normalized}.
\end{defi}

\begin{defi}[Rewritten Criterion]
\label{def:rewcrit}
Let $(r_i,r_j) \in G \times G$ be a critical pair. $\Spol(r_i,r_j)$ is \emph{rewritable} iff for $u_k r_k$, $k=i$ or $k=j$, there exist $r_v,r_w \in G$ such that
\begin{eqnarray*}
\ind(r_k) &=& \ind(\Spol(r_v,r_w)) \textrm{ and} \\ 
\Gamma\Big(\Sig\big((\Spol(r_v,r_w)\big)\Big) &\mid& u_k \Gamma\big(\Sig(r_k) \big)
\end{eqnarray*}
If there exist no such $r_v,r_w \in G$ then $\Spol(r_i,r_j)$ is called \emph{not rewritable}.
\end{defi}

\begin{thm}
\label{thm:main}
Let $\mathcal{L} \subset G \times G$ be such that for each pair $(r_i,r_j) \in \mathcal{L}$ $\Spol(r_i,r_j)$ is
\begin{enumerate}
\item normalized, and
\item not rewritable.
\end{enumerate}
Furthermore, if for each such pair $(r_i,r_j) \in \mathcal{L}$ $\Spol(r_i,r_j)$ has an admissible labeled $t$-representation or $\Spol(p_i,p_j)$ reduces to zero then $\poly(G)$ is a Gr\"obner basis of $I=\langle f_1,\dots,f_m \rangle$.
\end{thm}

\begin{proof}
Let $(r_i,r_j) \notin \mathcal{L}$. Then $\Spol(r_i,r_j)$ is either not normalized or rewritable. We have to show that all such S-Polynomials either have an admissible labeled $t$-representation for $t=\LCM\big(\HT(p_i),\HT(p_j)\big)$ or reduce to zero. \\
We can assume that $u_j \Sig(r_j) \precf u_i \Sig(r_i)$ and w.l.o.g. we can assume that in each case $u_i r_i$ is the admissible labeled polynomial detected by one or both of the two criteria (see Remark \ref{rem:proofcrit}).
For this let $r_i=(t_i \me_k, p_i)$.
\begin{enumerate}
\label{proof:casenotnorm}
\item Assume that $u_i r_i$ is not normalized. In this case there exists an element $r_\prev$ in $G$ with $\ind(r_\prev) > k$ and $\Gamma\big(u_i \Sig(r_i) \big) = u_i t_i = \lambda \HT(p_\prev)$ for some $\lambda \in \mathcal{T}$. This can be translated to a relation between two syzygies in $\KxnG$: We receive a principal syzygy given by $p_\prev$ and $f_k$, namely 
\begin{equation*}
\ms_{\prev,k} = p_\prev \me_k - f_k \me_\prev \in \KxnG. 
\end{equation*}
For $r_i$ there are two possibilities:
\begin{enumerate}
\item If $i \in \{1,\dots,m\}$ then we can construct a trivial syzygy $\ms_i = \me_i - \me_i$. Note that in this case $k=i$.
\label{proof:casetriv}
\item If $i \notin \{1,\dots,m\}$ then $r_i$ is the result of a reduction of an S-Polynomial, such that there exists a syzygy 
\begin{equation*}
\ms_i = \sum_{\ell = k}^{n_i} a_\ell^i \me_\ell - \me_i
\end{equation*}
where $n_i$ denotes the number of elements in the subsequent Gr\"obner basis $G$ before $r_i$ is added. It holds that $\MHTF(\ms_i) = \Sig(r_i)$. 
 
\end{enumerate}
Either way $\MHTF(u_i \ms_i) = \MHTF(\lambda \ms_{\prev,k})$ by construction and we can compute their difference:
\begin{align}
\label{proof:eqnotnorm}
\lambda \ms_{\prev,k} - u_i \ms_i &= \left(\lambda \LOT(p_\prev) - u_i \LOT(a_k^i) \right) \me_k + \sum_{\ell=k+1}^{n_i} a_\ell^i \me_\ell + \notag \\
                                  &\quad + \lambda f_k \me_\prev - u_i \me_i.
\end{align}
By construction 
\begin{align*}
\HT\big(\lambda \LOT(p_\prev) - u_i \LOT(a_k^i)\big) \Sig(r_k) &\precf u_i \Sig(r_i) \\
\HT(a_\ell^i) \Sig(r_\ell) &\precf u_i \Sig(r_i) \textrm{ for all } \ell \in \{k+1,\dots,n_i\} \\
\lambda \HT(f_k) \Sig(r_\prev) &\precf u_i \Sig(r_i).
\end{align*}
Note that in case \ref{proof:casetriv} $u_i \LOT(a_k^i)$ is zero.
As $\ms_i$ and $\ms_{\prev,k}$ are syzygies it holds that $v_G(u_i \ms_i - \lambda \ms_{\prev,k}) = 0$.
\item Assume that $u_i r_i$ is rewritable. In this case there exists an $\Spol(r_v,r_w)$ such that $\ind(\Spol(r_v,r_w)) = k$ and $\lambda \in \mathcal{T}$ such that $\lambda \Gamma\Big(\Sig\big((\Spol(r_v,r_w)\big)\Big) = \Gamma\big(u_k \Sig(r_k) \big)$.  
Again we can translate these data to a relationship between two syzygies. For $r_i$ we have the same possibilities as mentioned in the case of $u_ir_i$ not normalized above, in short:
\begin{enumerate}
\item If $i \in \{1,\dots,m\} \Rightarrow \ms_i = \me_i - \me_i$.
\item If $i \notin \{1,\dots,m\} \Rightarrow \ms_i = \sum_{\ell = k}^{n_i} a_\ell^i \me_\ell - \me_i$.  
\end{enumerate}
This time we also need to have a closer look at the syzygy given by $\Spol(r_v,r_w)$. Based on the implementation of the Rewritten Criterion in the $F_5$ algorithm $\Spol(r_v,r_w)$ is not rewritable, as otherwise $\Spol(r_i,r_j)$ would be detected by the S-Polyinomial which rewrites $\Spol(r_v,r_w)$. $\Spol(r_v,r_w)$ has been already or eventually will be reduced to a new element $r_\rew \in G$, so it has a $t$-representation for $t < \LCM\big(\HT(p_v),\HT(p_w)\big)$, or it has been reduced to zero w.r.t. $G$. In either way we receive a syzygy 
\begin{equation*}
\ms_{v,w} = \sum_{\ell=k}^{n_\rew} a_\ell^\rew \me_\ell - \alpha \me_\rew
\end{equation*}
where $n_\rew$ denotes the number of elements in the subsequent Gr\"obner basis $G$ before $r_\rew$ is possibly added. $\alpha=0$ if $\Spol(r_v,r_w)$ reduces to zero, and $\alpha=1$ otherwise. It holds that $\MHTF(\ms_{v,w}) = \Sig\big(\Spol(r_v,r_w)\big)$. \\
Analogously to the case of $u_ir_i$ being not normalized we compute the difference of the two syzygies $u_i \ms_i$ and $\lambda \ms_{v,w}$ which fulfill the relation $\MHTF(u_i \ms_i) = \MHTF(\lambda \ms_{v,w})$:
\begin{align}
\label{proof:eqrew}
\lambda \ms_{v,w} - u_i \ms_i &= \left( \lambda \LOT(a_k^\rew) - u_i \LOT(a_k^i)\right) \me_k + \sum_{\ell=k+1}^{n_\textrm{min}} (\lambda a_\ell^\rew - u_i a_\ell^i) \me_\ell \notag \\
                              &\quad + \sum_{\ell'=n_\textrm{min}+1}^{n_\textrm{max}} \lambda a_{\ell'}^\rew \me_{\ell'} - \lambda \alpha \me_\rew + u_i \me_i \notag \\
                              &= \left( \lambda \LOT(a_k^\rew) - u_i \LOT(a_k^i)\right) \me_k + \sum_{\ell=k+1}^{n_\textrm{max}} (\lambda a_\ell^\rew - u_i a_\ell^i) \me_\ell  \notag \\
			      &\quad - \lambda \alpha \me_\rew + u_i \me_i
\end{align}
where we define $n_\textrm{min}=\textrm{min}\{n_i,n_\rew\}$, $n_\textrm{max}=\textrm{max}\{n_i,n_\rew\}$. Note that in Equation (\ref{proof:eqrew}) 
\begin{eqnarray*}
a_\ell^i = 0 \textrm{ for } \ell \in \{n_i+1,\dots,n_\textrm{max} \} \textrm{ or} \\ 
a_\ell^\rew = 0 \textrm{ for } \ell \in \{ n_\rew+1,\dots,n_\textrm{max} \},
\end{eqnarray*}
depending on the relation of $n_i$ and $n_\rew$.
It holds that $v_G(\lambda \ms_{v,w} - u_i \ms_i) = 0$, moreover 
\begin{align*}
\HT\big( \lambda \LOT(a_k^\rew)- u_i \LOT(a_k^i)   \big) \Sig(r_k) &\precf u_i \Sig(r_i) \\
\HT \big(  \lambda a_\ell^\rew - u_i a_\ell^i  \big) \Sig(r_\ell) &\precf u_i \Sig(r_i) \textrm{ for all } \ell \in \{k+1,\dots,n_\textrm{max}\}.
\end{align*}
Note that $\lambda \Sig(r_\rew) =_{\textrm{F}} u_i \Sig(r_i)$ by construction.
\end{enumerate}
In both of the stated cases a new syzygy is built, we can summarize (\ref{proof:eqnotnorm}) and (\ref{proof:eqrew}) in one syzygy $\ms_\crit$: 
\begin{equation}
\label{proof:eqcrit}
\ms_\crit = \sum_{\ell=k}^{n_\textrm{max}} a_\ell \me_\ell - \mu \me_\crit + u_i \me_i
\end{equation}
where $\HT(a_\ell) \Sig(r_\ell) \precf u_i \Sig(r_i)$ for all $\ell \in \{k,\dots,n_\textrm{max}\}$ and $\mu \Sig(r_\crit) \preceqf u_i \Sig(r_i)$. \\
As $v_G(\ms_\crit) = 0$ every head term of each evaluated element from $\ms_\crit$ needs to be reduced. 
Thus we find two elements $a_\ell \me_\ell$ and $a_{\ell'} \me_{\ell'}$ in $\ms_\crit$ such that 
\begin{equation*}
\HT\big(a_\ell v_G(\me_\ell)\big) = \HT\big(a_{\ell'} v_G(\me_{\ell'})\big). 
\end{equation*}
This corresponds to a multiple of $\Spol(r_\ell,r_{\ell'})$ where both, $u_\ell r_\ell$ and $u_{\ell'} r_{\ell'}$ have a signature lower or equal to the one of $u_i r_i$ w.r.t. $\precf$. 
These S-Polynomials are either rewritable/not normalized and can be rewritten in the same way without increasing their signatures or head terms, or they reduce to an element $r_\red \in G$ such that $\Sig(r_\red) = \Sig\big(\Spol(r_\ell,r_{\ell'})\big)$ and $\HT(p_\red) < u_\ell \HT(p_\ell)$, or they reduce to zero w.r.t. $G$. 
This building, reducing and deleting of new S-Polynomials stops after a finite number of steps because of the finiteness of the polynomials and their signatures. \\
We stop this process when we have found an element $u_{\ell_0} \me_{\ell_0}$ in $\ms_\crit$  such that 
\begin{equation*}
u_{\ell_0} \HT\left(v_G(\me_{\ell_0})\right) = u_i \HT(p_i).
\end{equation*}
Thus we have found a multiple of $\Spol(r_i,r_{\ell_0})$. We have to distinguish the following cases:
\begin{enumerate}
\item If $u_{\ell_0} r_{\ell_0} \neq u_j r_j$ then we can represent $\ms_\crit$ from Equation (\ref{proof:eqcrit}) by 
\begin{equation*}
\ms_\crit = \sum_{\ell=k}^{n'} b_\ell \me_\ell - u_{\ell_0} \me_{\ell_0} + u_i \me_i
\end{equation*}
where $\HT(b_\ell p_\ell) < u_i \HT(p_i)$ for all $\ell \in \{k,\dots,n' \}$ and $n' = n_\textrm{max}+1$. Note that we can assume $\mu \me_\crit$ to be part of the sum. Using the evaluation we get
\begin{align*}
0 &= \sum_{\ell = k}^{n'} b_\ell p_\ell - u_{\ell_0} p_{\ell_0} + u_i p_i \\
0 &= \sum_{\ell = k}^{n'} b_\ell p_\ell + \nu_1 \Spol(p_i,p_{\ell_0}) \textrm{ for some } \nu_1 \in \mathcal{T} \\
&\Rightarrow \nu_1 \Spol(p_i,p_{\ell_0}) =- \sum_{\ell = k}^{n'} b_\ell p_\ell.
\end{align*}
From this equation we receive an admissible labeled $t_1$-representation for $t_1 = \nu_1 \LCM\big(\HT(p_i),\HT(p_{\ell_0})\big)$.\\
On the other hand we notice that $u_j \HT(p_j) = u_{\ell_0} \HT(p_{\ell_0})$ and thus there exists a multiple $\nu_2 \Spol(r_{\ell_0},r_j)$. This S-Polynomial is already reduced (possibly to zero) w.r.t. $G$ or detected by the two criteria and can be rewritten in the same way, where this process has to stop after a finite number of times. In any case it will be investigated in the $F_5$ algorithm and we can assume it to reduce to zero or to have an admissible labeled $t_2$-representation for $t_2 = \nu_2 \LCM\big(\HT(p_{\ell_0}),\HT(p_j)\big)$. Altogether we have a relation between three S-Polynomials:
\begin{equation*}
\Spol(p_i,p_j) = \nu_1 \Spol(p_i,p_{\ell_0}) + \nu_2 \Spol(p_{\ell_0},p_j).
\end{equation*}
Possibly there are further reductions of these S-Polynomials or detections by the two criteria, but all of these do not increase the signature and do lower the head term of the polynomials. \\
Assuming the reduction of $\Spol(r_i,r_{\ell_0})$ and $\Spol(r_{\ell_0},r_j)$ and noting the signatures of all elements which are $\preceqf u_i \Sig(r_i)$ we have an admissible labeled $t$-representation of $\Spol(r_i,r_j)$.

\item If $u_{\ell_0} r_{\ell_0} = u_j r_j$ then the represention of $\ms_\crit$ is given by
\begin{equation*}
\ms_\crit = \sum_{\ell=k}^{n'} b_\ell \me_\ell -u_j \me_j + u_i \me_i
\end{equation*}
where $\HT(b_\ell p_\ell) < u_i \HT(p_i)$ for all $\ell \in \{k,\dots,n'\}$ and $n'=n_\textrm{max}+1$. Again using the evaluation we get
\begin{align*}
0 &= \sum_{\ell = k}^{n'} b_\ell p_\ell -u_j p_j + u_i p_i \\
0 &= \sum_{\ell = k}^{n'} b_\ell p_\ell + \Spol(p_i,p_j) \\
&\Rightarrow \Spol(p_i,p_j) =- \sum_{\ell = k}^{n'} b_\ell p_\ell
\end{align*}
Again assuming further reductions or detections by the two criteria inside $\sum_{\ell=k}^{n'} b_\ell p_\ell$ from this equality we directly receive an admissible labeled $t$-representation of $\Spol(r_i,r_j)$ for $t = \LCM\big(\HT(p_i),\HT(p_j)\big)$. 
\end{enumerate}
Thus $\poly(G)$ is a Gr\"obner basis for $I$.
\end{proof}

\begin{rem} \
\label{rem:proofcrit}
\begin{enumerate}
\item In the case of $u_i r_i$ being rewritable by $\lambda r_\rew$ it is possible that $u_{\ell_0} r_{\ell_0} = \lambda r_\rew$ also. Then by the same construction as stated in the proof we get
\begin{equation*}
\Spol(p_i,p_\rew) = -\sum_{\ell = k}^{n'} b_\ell p_\ell.
\end{equation*}
In this case $\HT(b_\ell) \Sig(r_\ell) \precf u_i \Sig(r_i) = \lambda \Sig(r_\rew)$ for all $\ell \in \{k,\dots,n'\}$. Thus $\Spol(r_i,r_\rew)$ can be rewritten by a linear combination of elements in $G$ with lower signatures, thus we have found an admissible labeled $t$-representation of $\Spol(r_i,r_\rew)$ for $t = \LCM\big(\HT(p_i),\HT(p_\rew)\big)$. \\
Note that this also includes the case where $u_{\ell_0} r_{\ell_0} = u_j r_j = \lambda r_\rew$. 
\item In the case $u_{\ell_0} r_{\ell_0} \neq u_j r_j$ we denote the second computed S-Polynomial 
\begin{equation*}
\Spol(r_{\ell_0},r_j) = u_{j, {\ell_0}} r_{\ell_0} - u_{ {\ell_0},j} r_j. 
\end{equation*}
Of course it can happen that  $u_{j, {\ell_0}} \Sig(r_{\ell_0}) \precf u_{ {\ell_0},j} \Sig(r_j)$. In this case we would compute $\Spol(r_j,r_{\ell_0})$, but this would just lead to a difference in sign and would not change the arguments of the proof, hence we have omitted the distinction between these  two possibilities above.

\item Setting $n'=n_\textrm{max}+1$ is only necessary in the case where $n_\rew = \textrm{max}\{n_i,n_\rew\}$ and $u_{\ell_0} r_{\ell_0} \neq \lambda r_\rew$, i.e. if $\lambda p_\rew$ is inside $\sum_{\ell=k}^{n'} b_\ell p_\ell$. Since $n_\textrm{max}$ denotes the number of elements before $r_\rew$ enters $G$ in this case, $n' = \rew$. In all other cases $b_{n'} = 0$.
\item When building S-Polynomials inside $\ms_\crit$ until we end up with $u_{\ell_0} \me_{\ell_0}$ the signatures do not increase. 
This is due to the $F_5$ algorithm: If there is a reductor $r_\red$ of an element $r_\spp$, where $r_\spp$ denotes the possibly already reduced S-Polynomial investigated by $F_5$ in this step, 
such that there exists $u_\red \in \mathcal{T}$ where $u_\red \HT(p_\red) = \HT(p_\spp)$ and $u_\red \Sig(r_\red) \succf \Sig(r_\spp)$ than two elements will be returned by the procedure \verb|TopReduction|: The (in this step of the algorithm) not top-reduced element $r_\spp$ for which the reductor was found and a new S-Polynomial $\Spol(r_\red,r_\spp)$ with $\Sig\big(\Spol(r_\red,r_\spp)\big) = u_red \Sig(r_\red)$. From this point on both elements are investigated separately from each other for further reductions. 
So if we have defined an S-Polynomial in the beginning there is no change of its signature in the whole reduction process, and thus there is no increasing of the signatures in the proof.

\item Note that if we assume $u_j r_j$ to be not normalized/rewritable in the beginning instead of $u_i r_i$ the proof would work exactly the same way, it would be even easier since 
\begin{equation*}
u_\ell \Sig(r_\ell) \preceqf u_j \Sig(r_j) \precf u_i \Sig(r_i) \textrm{ for all } \ell \in \{k,\dots,n_\textrm{max}\},
\end{equation*}
and due to this relation of the signatures it cannot happen that $u_{\ell_0} r_{\ell_0} = u_i r_i$.

\end{enumerate}

\end{rem}

\section{Examples Of The Criteria Used In The $F_5$ Algorithm}
\label{sec:examples}
In this section we give some examples of the $F_5$ Criterion and the Rewritten Criterion. For this purpose we use the example given in both \cite{traverso} Section 7 and \cite{f5} Section 8. We will not state the whole computations and refer to the afore-mentioned papers for more details.\\
Note that we do not explain in detail the difference between the computations done in both papers, but we show the critical pair the Rewritten Criterion detects to be useless whereas the criterion of M\"oller, Traverso and Mora stated in \cite{traverso} does not detect it. \\
The proof of Theorem \ref{thm:main} gives us a constructive explanation of the criteria which we use in every of the following computations.\\
In this example we want to compute the Gr\"obner basis of the ideal $I$ given by
\begin{eqnarray*}
f_1 &=& yz^3 - x^2t^2\\
f_2 &=& xz^2 - y^2t \\
f_3 &=& x^2y - z^2t 
\end{eqnarray*}
in $\mathbb{Q}[x,y,z,t]$ with degree reverse lexicographical ordering $x>y>z>t$. As agreed in Convention \ref{conv:embedding} $r_i := (\me_i,f_i)$ for $i \in \{1,2,3\}$.  

\subsection{Some Examples Of The Rewritten Criterion}
\label{sec:exrew}
We give three examples of the Rewritten Criterion. In the first example we rewrite a multiple of an element from $\{f_1,\dots,f_m\}$, in the second one we generalize this attempt for arbitrary elements in $G$ during the computation of $F_5$. In the last example we see that the Rewritten Criterion also covers direct paraphrases in which we get an admissible labeled $t$-representation of the investigated S-Polynomial immediately.
\begin{enumerate}
\item $P_8 = x^2 r_1 - z^3 r_3$ is rewritable since $x^2 \Sig(r_1) = x \Sig(r_6)$. 
Thus for the computation of $r_6$ we have received a syzygy $\mathbf{s}_6 = x \mathbf{e}_1 - yz \mathbf{e}_2 - \mathbf{e}_6$ such that $x \MHTF(\mathbf{s}_6) = x^2 \mathbf{e}_1$. For $r_1$ we get an trivial syzygy $\ms_1 = \me_1 - \me_1$. Computing the difference of multiples of these syzygies we get
\begin{equation*}
x^2 \ms_1 + x \ms_6 = x^2 \me_1 - xyz \me_2 - x \me_6 
\end{equation*}
where $x^2 \HT(p_1) = xyz \HT(p_2)$. So when evaluating we get a reduction of a multiple of $\Spol(p_1,p_2)$:
\begin{equation*}
x \Spol(p_1,p_2) = x^2 p_1 - xyz p_2 = x p_6
\end{equation*}
where $x \Sig(r_6) =_\textrm{F} x^2 \Sig(r_1)$. On the other hand we compute a second multiple of an S-Polynomial with $xyz p_2$ and $z^3 p_3$ $z \Spol(p_2,p_3)$ which is already reduced to the element $z p_4$. Using the relation
\begin{equation*}
\Spol(p_1,p_3) = x \Spol(p_1,p_2) + z \Spol(p_2,p_3)
\end{equation*}
$\Spol(r_1,r_2)$ has an admissible labeled $t$-representation.

\item $P_{15} = xz r_6 - y^3t r_2$ is rewritable since $xz \Sig(r_6) = z \Sig(r_7)$. Again we have 
\begin{eqnarray*}
\mathbf{s}_7 &=& x \mathbf{e}_6 - z \mathbf{e}_4 - \mathbf{e}_7, \\
\mathbf{s}_6 &=& \mathbf{e}_6 - \mathbf{e}_6.
\end{eqnarray*}
To get the related S-Polynomials we compute 
\begin{eqnarray*}
xz \mathbf{s}_6 + z \mathbf{s}_7 &=& xz \me_6 - xz \me_6 + xz \me_6 - z^2 \me_4 - z \me_7 \\
                                 &=& xz \me_6 -z^2 \me_4 - z \me_7
\end{eqnarray*}
The next reduction would be done with $xz \me_6$ resp. $xz p_6$. Thus we receive that $\HT(x^2 p_4) = \HT(xz p_6)$ which leads to $z \Spol(p_6,p_4)$. Clearly we also get an S-Polynomial for $y^3t p_2$, namely $\Spol(p_4,p_2)$ and together we receive
\begin{equation*}
\Spol(p_6,p_2) = z \Spol(p_6,p_4) + \Spol(p_4,p_2),
\end{equation*}
an admissible labeled $t$-representation of $\Spol(r_6,r_2)$.
\item $P_{18} = x r_8 - y^2 t r_4$ is rewritable since $x \Sig(r_8) = z \Sig(r_9)$. 
Note that we do not use the completely reduced polynomial $r_9$ which Faug\`{e}re computes in the given example in \cite{f5} but the reduction given from the $F_5$ algorithm, i.e. $r_9 = (x^3 \me_1,-x^5t^2+y^2z^3t^2)$. 
We have
\begin{eqnarray*}
\ms_8 &= z \me_7 - \me_5 - \me_8 \\
\ms_9 &= x \me_7 - z^3t \me_2 - \me_9
\end{eqnarray*}
In the same way we compute
\begin{eqnarray*}
x \ms_8 - z \ms_9 = z^4t \me_2 - x \me_5 - x \me_8 + z \me_9.
\end{eqnarray*}
The evaluation of the first two elements on the right-hand side of the equation is equal to $-\Spol(p_5,p_2)$ which can be rewritten as $y^2t p_4$ such that we get that 
\begin{align*}
v_G(x \ms_8) - v_G(z \ms_9) = v_G(y^2t \me_4) - v_G(x \me_8) + v_G(z \me_9) &= 0 \\
                                      -\Spol(p_8,p_4)+ z p_9 &= 0
\end{align*}
such that $\Spol(r_8,r_4)$ is useless for further computations.
\end{enumerate}

\begin{rem}
Note that the last example above is the one reduction to zero which is not detected in \cite{traverso}. Using a criterion for detecting syzygies, i.e. relations between S-Polynomials, too, M\"oller, Traverso and Mora are using other descriptions of the polynomials and do not give the polynomials a label or signature. The syzygies and polynomials computed during the algorithm are strictly separated in their attempt, whereas in Faug\`{e}re's idea the syzygies do not need to be computed, as their module head terms can be deduced by the signatures of the computed polynomials. 
\end{rem}

\subsection{Some Examples Of The $F_5$ Criterion}
In the following three examples of the $F_5$ Criterion are shown. 
The first example explains the direct paraphrase in which we can find an admissible $t$-represenation of the investigated S-Polynomial immediately. In the second example we end with a relation between the S-Polynomial in question and two other S-Polynomials, one of them is already detected to be not normalized (first example), the other investigated as the third example.
\begin{enumerate}

\item $P_{11} = z^2 r_6 - y^2 t r_1$ is not normalized since $z^2 \Sig(r_6) = xz^2 \me_1$ and $xz^2=\HT(r_2)$.
So we compute the syzygies 
\begin{eqnarray*}
\ms_{1,2} &=& r_2 \me_1 - r_1 \me_2 \\
          &=& x z^2 \me_1 - y^2t \me_1 - yz^3 \me_2 + x^2t^2 \me_2\\
z^2 \ms_6 &=& xz^2 \me_1 - yz^3 \me_2 - z^2 \me_6.
\end{eqnarray*}
In the same way as in Section \ref{sec:exrew} we compute their difference to see the relations of S-Polynomials:
\begin{eqnarray*}
z^2 \ms_6 - \ms_{1,2} &=& y^2 t \me_1 - x^2t^2 \me_2 - z^2 \me_6, \textrm{ where} \\
y^2 t\HT( p_1) &=& y^3 z^3 t = z^2 \HT(p_6), \textrm{ and} \\
x^2 t^2 \HT(p_2) &<& y^3 z^3 t.
\end{eqnarray*}
Thus we receive the following relation of polynomials when evaluating the difference of syzygies above:
\begin{align*}
v_G(z^2 \ms_6) - v_G(\ms_{1,2}) = v_G(y^2 t \me_1) - v_G(x^2t^2 \me_2) - v_G(z^2 \me_6) &= 0 \\
                                                   - \Spol(p_6,p_1) - x^2 t^2 p_2  &= 0.
\end{align*}
It follows that $\Spol(p_6,p_1)$ is reduced to zero by $x^2t^2 p_2$.
\item Another pair which is deleted by the $F_5$ Criterion is the pair $(r_7,r_6)$ which corresponds to $\Spol(r_7,r_6)= (x^2 y^3 \me_1,y^3 r_7 - z^4 r_6)$. Since $y^3 \Sig(r_7) = x^2 y^3 \me_1$ and $x^2y^3 = y^2 \HT(r_3)$ it is not normalized. Note that in this example also $z^4 r_6$ is not normalized since $z^4 \Sig(r_6) = x z^4 \me_1$ and $x z^4 = z^2 \me_2$. \\
Again we compute two syzygies we want to subtract from each other
\begin{eqnarray*}
y^2\ms_{1,3} &=& y^2r_3 \me_1 - y^2r_1 \me_3 \\
             &=& x^2y^3 \me_1 - y^2z^2t \me_1 - y^3z^3 \me_3 + y^2x^2t^2 \me_3\\
y^3 \ms_7 &=& x y^3 \me_6 - y^3 z \me_4 - y^3 \me_7 \\
          &=& x^2 y^3 \me_1 - x y^4 z \me_2 - y^3 z \me_4 - y^3 \me_7.
\end{eqnarray*}
This leads to the computation of the difference of both syzygies
\begin{eqnarray*}
y^3 \ms_7 - y^2 \ms_{1,3} = y^2 z^2 t \me_1 - x y^4 z \me_2 - y^3 z \me_4 - y^3 \me_7 - y^3z^3 \me_3 + x^2y^2t^2 \me_3
\end{eqnarray*}
where some more S-Polynomials are computed but already at this point one can see that $y^2z^2t \HT(p_1) = y^3 \HT(p_7)$ and we get $-y^2 \Spol(p_7,p_1)$. Again from the construction we also can compute that $y^2z^2t \HT(p_2) = z^4 \HT(p_6)$ and we get $z^2 \Spol(p_6,p_1)$. \\
$\Spol(r_6,r_1)$ was investigated in Case $(a)$, $\Spol(r_7,r_1)$ is also deleted by the $F_5$ Criterion, so we have a closer look at it in the following example. We get 
\begin{equation*}
\Spol(p_7,p_6) = y^2 \Spol(p_7,p_1) - z^2 \Spol(p_6,p_1),
\end{equation*}
an admissible labeled $t$-representation of $\Spol(r_7,r_6)$.

\item $\Spol(r_7,r_1) = (x^2y \me_1,y r_7 - z^2t r_1)$ is not normalized since $y \Sig(r_7) = x^2 y \me_1$ and $x^2y = \HT(r_3)$. 
We have already computed the two syzygies
\begin{eqnarray*}
\ms_{1,3} &=& r_3 \me_1 - r_1 \me_3 = x^2y \me_1 - z^2t \me_1 - yz^3 \me_3 + x^2t^2 \me_3,\\ 
y \ms_7 &=& x^2y \me_1 - xy^2z \me_2 - yz \me_4 - y \me_7.
\end{eqnarray*}
So we get 
\begin{equation*}
y \ms_7 - \ms_{1,3} = x^2 y \me_1 - xy^2z \me_2 - yz \me_4 - y \me_7 - x^2 y \me_1 + z^2t \me_1 +yz^3 \me_3 - x^2t^2 \me_3.
\end{equation*}
Firstly $yz\Spol(p_2,p_3)$ is built which cancels with $yzp_4$ such that in the end we get
\begin{align*}
v_G(y \ms_7) - v_G(\ms_{1,3}) = -v_G(y \me_7) + v_G(z^2t \me_1) - v_G(x^2t^2 \me_3) &= 0 \\
                                              -\Spol(p_7,p_1) - x^2 t^2 p_3    &= 0.
\end{align*}
Thus $\Spol(r_7,r_1)$ is useless and can be deleted.
\end{enumerate}

\section{Improving The $F_5$ Criterion?}
\label{sec:improvement}
Having a closer look at Equation (\ref{proof:eqrew}) in the proof of Theorem \ref{thm:main} we note that instead of the not normalized case we have $\lambda \Sig(r_\rew) =_{\textrm{F}} u_i \Sig(r_i)$ in the rewritable case, so we do not need to require after cancellation of the $\MHT$s that all elements besides $u_i \me_i$ have signature lower than $u_i \Sig(r_i)$ w.r.t. $\precf$, it is enough to claim that there is no element in the syzygy having a signature bigger than $u_i \Sig(r_i)$ w.r.t. $\precf$ . 
Thus the question arises if the requirement of the $F_5$ Criterion that $\ind(r_\prev) < \ind(r_i)$ is too restrictive. \\
In the following we give a generalized definition of the $F_5$ Criterion due to the assumption stated above and prove that this does not give any improvement.

\begin{defi}[Improved $F_5$ Criterion]
\label{def:f5critimp}
Let $(r_i,r_j) \in G \times G$ be a critical pair. $\Spol(r_i,r_j)$ is \emph{not completely normalized} iff for $u_k r_k$ where $k=i$ or $k=j$ there exists $r_\prev \in G$ such that one of the following cases holds: 
\begin{enumerate}
\item $\Spol(r_i,r_j)$ is not normalized.
\item There exists $\lambda \in \mathcal{T}$ such that 
\begin{eqnarray*}
\ind(r_\prev) &=& \ind(r_k) =: k_0 \\
\lambda \HT(p_\prev) &=& u_k \Gamma\big(\Sig(r_k)\big) \\
\HT(f_{k_0}) \Gamma\big(\Sig(r_\prev) \big) &<& \HT(p_\prev).
\end{eqnarray*}
\end{enumerate}
If there exists no such $r_\prev \in G$ then $\Spol(r_i,r_j)$ is \emph{completely normalized}.
\end{defi}

\begin{rem}
Note that from the discussion in the beginning of this section it seems to make sense to generalize the last inequality in part (b) of Definition \ref{def:f5critimp} to
\begin{equation*}
\HT(f_{k_0}) \Gamma \big( \Sig(r_\prev) \big) \leq \HT(p_\prev).
\end{equation*}
In the proof of the following lemma we show that this equality exists, but it is a trivial case which cannot be used as a criterion to detect useless critical pairs while computing Gr\"obner bases. See Remark \ref{rem:f5impequal} for a more detailed explanation.
\end{rem}

Next we show that the Improved $F_5$ Criterion detects the same critical pairs than the $F_5$ Criterion. Thus Defintion \ref{def:f5critimp} is no improvement of Definition \ref{def:f5crit}.

\begin{lem}
\label{lem:contraf5imp}
Let $(r_i,r_j) \in G \times G$ be a pair of admissible labeled polynomials, then $\Spol(r_i,r_j)$ is
\begin{equation*}
\textrm{normalized} \Leftrightarrow \textrm{completely normalized}
\end{equation*}
\end{lem}

\begin{proof}
We have to show that there exist no $\Spol(r_i,r_j) \in G \times G$ and $r_\prev \in G$ such that part (b) of Definition \ref{def:f5critimp} is fulfilled. \\
Assume the contrary, for $k=i$ or $k=j$ let $\ind(r_\prev) = \ind(r_k) = k_0$, $\lambda \in \mathcal{T}$ such that $\lambda \HT(p_\prev) = \Gamma \big(\Sig(r_k) \big)$ and $\HT(f_{k_0}) \Gamma\big(\Sig(r_\prev) \big) < \HT(p_\prev)$. We assume that $r_\prev$ fulfills only part (b) of Definition \ref{def:f5critimp}. We show that there exists no such element in $G$. For this we have to distinguish two cases:
\begin{enumerate}
\item If $p_\prev \in \{f_1,\dots,f_m\}$ then $p_\prev = f_{k_0}$ as $\ind(r_\prev)= k_0$. Furthermore $\Gamma \big( \Sig(r_\prev) \big) = 1$. By our assumptions 
\begin{eqnarray*}
\HT(f_{k_0}) \Gamma\big(\Sig(r_\prev)\big) &<& \HT(p_\prev) \\
\Rightarrow \HT(f_{k_0}) \cdot 1 &<& \HT(f_{k_0})
\end{eqnarray*}
which is a contradiction.
\item If $p_\prev \notin \{f_1,\dots,f_m\}$ then 
\begin{enumerate}
\item $p_\prev$ is the reduction of $\Spol(f_{k_0},p_\ell)$ for some $r_\ell \in G$ such that it holds that $\ind(r_\ell) > k_0$. Let $u_{k_0}=\frac{\LCM\big(\HT(f_{k_0}),\HT(p_\ell)\big)}{\HT(f_{k_0})}$ then it follows that $u_{k_0} = \Gamma\big(\Sig(r_\prev)\big)$ and  
\begin{equation*}
\HT(f_{k_0}) \Gamma \big(\Sig(r_\prev) \big) > \HT(p_\prev)
\end{equation*}
as the head terms of $\Spol(f_{k_0},p_\ell) = \HT(f_{k_0}) \Gamma\big(\Sig(r_\prev)\big)$ cancel during the reduction step.
\item $p_\prev$ is the reduction of $\Spol(p_u,p_v)$ for $p_u,p_v \in G$. Inductively using the same argument as above we receive that 
\begin{equation*}
\HT(f_{k_0}) \Gamma \big(\Sig(r_\prev) \big) > \HT(p_\prev).
\end{equation*}
\end{enumerate}
Thus both subcases contradict our assumptions about $p_\prev$.
\end{enumerate}
Thus we have shown that there exists no admissible labeled polynomial $r_\prev \in G$ which fulfills part (b) of Definition \ref{def:f5critimp}.
\end{proof}
\begin{rem}
\label{rem:f5impequal}
\begin{enumerate}
\item From part (a) of the proof of Lemma \ref{lem:contraf5imp} we see that the only possible case which would still hold the condition on the syzygies of the proof of the main theorem, namely no signature bigger than the one of the not normalized/rewritable element, is
\begin{equation*}
\HT(f_{k_0}) \Gamma \big( \Sig(r_\prev) \big) = \HT(p_\prev).
\end{equation*}
Note that this is only the case when $p_\prev=f_{k_0}$ such that it leads to a trivial, not principal, syzygy, i.e.
\begin{eqnarray*}
\ms_{\prev,k_0} &=& p_\prev \me_{k_0} - f_{k_0} \me_\prev \\
                &=& f_{k_0} \me_{k_0} - f_{k_0} \me_{k_0} \\
                &=& 0 \in \Kxm.
\end{eqnarray*}
It follows that we do not receive a syzygy to compute relations of S-Polynomials and we cannot delete $\Spol(r_i,r_j)$ from the computations of $G$ without any other detection of further criteria.
\item From the point of view that the $F_5$ Criterion computes principal syzygies in $\Kxm$ it is easy to see that the criterion cannot be generalized relaxing the requirement on the index of $r_\prev$, as a principal syzygy with two elements of the same index will always end up in the trivial case stated above.
\end{enumerate}
\end{rem}

We have shown that the $F_5$ Criterion cannot be generalized in the sense of relaxing the condition on the indices.

\newpage
\begin{appendix}
\section{Implementation in \textsc{Singular}}
This appendix discusses another result of John Perry's and the author's joint work on the $F_5$ Algorithm, a freely-available library for the open-source computer algebra system \textsc{Singular}. 
\subsection{Sources}
This \verb|f5_library.lib| is an implementation of a slightly improved $F_5$ Algorithm in \textsc{Singular}. 
You can get it here:
\begin{center}
\verb|http://www.math.usm.edu/perry/Research/f5_library.lib|.
\end{center}
This library is implemented in the interpreted language in \textsc{Singular}, thus it is slow, but useful for testing the algorithms behaviour. You should also download a second library, \verb|f5ex.lib|, which consists of lots of precasted examples:
\begin{center}
\verb|http://www.math.usm.edu/perry/Research/f5ex.lib|.
\end{center}
A kernel implementation of $F_5$ in \textsc{Singular} is in preparation by the author. For more information about \textsc{Singular} visit
\begin{center}
\verb|http://www.singular.uni-kl.de/index.html|.
\end{center}
A good introduction to \textsc{Singular} and its applications in commutative algebra resp. algebraic geometry can be found in \cite{singular}.
\subsection{Using the Implementation}
The usage of \verb|f5_library.lib| is best explained in a little example: 
Let us assume the computation of the example given in Section 8 in \cite{f5}. Once \textsc{Singular} is started, it awaits an input after the prompt ``>''. 
Every statement has to be terminated by ``;''. Firstly we have to link the two above mentioned libraries to \textsc{Singular}, for this copy both libraries in your \textsc{Singular} folder. As \verb|f5ex.lib| is called internally by \verb|f5_library.lib| it is enough to link this one. 
The ideal to be computed can be generated by the command \verb|fmtm()|, which defines a basering and the ideal \verb|i|. 
In the following the output of \textsc{Singular} is accentuated by ``\verb|==>|''.
The following steps should be self-explanatory, otherwise use the online manual available at
\begin{center}
\verb|http://www.singular.uni-kl.de/Manual/latest/index.htm|.
\end{center}
\begin{verbatim}
LIB ``f5_library.lib'';
fmtm();
i;
==>i[1]=yz3-x2t2
==>i[2]=xz2-y2t
==>i[3]=x2y-z2t
ideal g;
g = basis(i);
==> cpu time for gb computation: 70/1000 sec
g;
==>g[1]=xz2-y2t
==>g[2]=x2y-z2t
==>g[3]=yz3-x2t2
==>g[4]=y3zt-x3t2
==>g[5]=xy3t-z4t
==>g[6]=z5t-x4t2
==>g[7]=y5t2-x4zt2
==>g[8]=x5t2-z2t5
\end{verbatim}
Typing \verb|help f5_library.lib;| resp. \verb|help f5ex.lib;| one gets more information about implemented procedures and their usage. \\
Moreover, there is an Gr\"obner basis algorithm implemented using the methods and ideas for detecting useless pairs given by Gebauer and M\"oller in \cite{gm} for the purpose of comparing both algorithms.
One can use it in the same way as explained above, changing \verb|basis()| to \verb|gm_basis()|. 

\end{appendix}


\begin{thebibliography}{{Gas}08}

\bibitem[Fau02]{f5}
J.C. Faug\`{e}re.
\newblock {A new efficient algorithm for computing Gr{\"o}bner bases without
  reduction to zero(F5)}.
\newblock {\em {Symbolic and Algebraic Computation, Proc. Conferenz ISSAC
  2002}}, pages 75--83, 2002.

\bibitem[{Gas}08]{gash}
{Gash, J.M.}
\newblock {On Efficient Computations of Grobner Bases}.
\newblock {\em {Dissertation}}, 2008.

\bibitem[GM88]{gm}
{Gebauer, R.} and {M\"oller, H.M.}
\newblock {On an Installation of Buchberger's Algorithm}.
\newblock {\em {Journal of Symbolic Computation, 6(2 and 3)}}, pages 275--286,
  1988.

\bibitem[GP02]{singular}
{Greuel, G.-M.} and {Pfister, G.}
\newblock {\em {A \textsc{Singular} Introduction to Commutative Algebra}}.
\newblock {Springer Verlag}, 2002.

\bibitem[MTM92]{traverso}
{M\"oller, H.M.}, {Traverso, C.}, and {Mora, T.}
\newblock {Gr\"obner bases computation using syzygies}.
\newblock {\em {ISSAC 92: Papers from the International Symposium on Symbolic
  and Algebraic Computation}}, pages 320--328, 1992.

\bibitem[{Ste}05]{f5rev}
{Stegers, Till}.
\newblock {Faug\`{e}re's F5 Algorithm Revisited}.
\newblock {\em {Thesis for the degreee of Diplom-Mathematiker}}, 2005.

\end{thebibliography}
\end{document}